\documentclass{llncs}
\usepackage{makeidx}
\usepackage[english]{babel}
\usepackage{amssymb,amsmath}
\usepackage{graphicx}
\def\bPsi{\boldsymbol\Psi}

\DeclareMathOperator{\Rea}{Re}
\DeclareMathOperator{\Ima}{Im}
\DeclareMathOperator{\dv}{div}
\begin{document}
\mainmatter
\title{On Error Estimates of the Crank-Nicolson-Polylinear Finite Element Method with the Discrete TBC for
the Generalized Schr\"odinger Equation in an Unbounded Parallelepiped}
\author{Alexander Zlotnik}
\institute{Department of Higher Mathematics at Faculty of Economic Sciences,\\
National Research University Higher School of Economics,\\
Myasnitskaya 20, 101000 Moscow, Russia
}
\maketitle
\begin{abstract}
We deal with an initial-boundary value problem for the generalized time-dependent Schr\"odinger equation with variable coefficients in an unbounded $n$--dimensional parallelepiped ($n\geq 1$).
To solve it, the Crank-Nicolson in time and the polylinear finite element in space method with the discrete transpa\-rent boundary conditions is considered.
We present its stability properties and derive new error estimates $O(\tau^2+|h|^2)$ uniformly in time in $L^2$ space norm, for $n\geq 1$, and mesh $H^1$ space norm, for $1\leq n\leq 3$ (a superconvergence result), under the Sobolev-type assumptions on the initial function.
Such estimates are proved for methods with the discrete TBCs for the first time.

\end{abstract}

\section{Introduction}
The linear time-dependent Schr\"odinger equation is the key one in many physical fields. It should be often solved in unbounded space domains.
A number of approaches were developed to deal with such problems using approximate transparent boundary conditions (TBCs) at the artificial boundaries.
\par Among the best methods of such kind are those using the so-called discrete TBCs remarkable by the clear mathematical background and the corresponding rigorous stability results in theory as well as the complete absence of spurious reflections in practice.
They first were constructed and studied for the standard Crank-Nicolson in time finite-difference schemes, see \cite{AES03,EA01} and also \cite{DZ06,DZ07}, in the cases of the infinite or semi-infinite axis and strip.
Later families of finite-difference schemes with general space averages were treated in \cite{DZZ09,ZZ11,IZ11}.
In particular, they include the linear and bilinear FEMs in space.
\par In this paper, we consider the Crank-Nicolson-polylinear FEM in an unbounded $n$--dimensional parallelepiped ($n\geq 1$), present results on its stability with respect to the initial data and a free term as well as on exploiting the discrete TBCs and mainly derive the corresponding new error estimates $O(\tau^2+|h|^2)$ uniform in time and in $L^2$ norm (for $n\geq 1$) and mesh $H^1$ norm (for $1\leq n\leq 3$) in space under the Sobolev-type assumptions on the initial function. The latter estimate is a superconvergence result.
Such estimates are proved for the methods with the discrete TBCs for the first time.
Importantly, the error estimates contain no mesh steps in negative powers like for other approximate TBCs, see \cite{JW08,JW10}, that is one more advantage of using the discrete TBCs.

\section{The IBVP and numerical methods to solve it}

We deal with the initial-boundary value problem (IBVP) for the time-dependent generalized Schr\"o\-din\-ger equation with $n\geq 1$ space variables
\begin{gather}
 i\hbar\rho D_t\psi={\mathcal H}\psi:=-\textstyle{\frac{\hbar^{\,2}}{2}}\dv (B\nabla\psi)+V\psi\ \
 \text{on}\ \ \Pi\times (0,T),
\label{eq:se}\\
 \left.\psi\right|_{\partial\Pi}=0,\ \ \psi|_{t=0}=\psi^0(x)\ \ \text{on}\ \ \Pi.
\label{eq:ic}
\end{gather}
Hereafter $\psi=\psi(x,t)$ is the complex-valued unknown wave function,
$i$ is the imaginary unit and $\hbar>0$ is a physical constant.
The $x=(x_1,\ldots,x_n)$-depending coefficients $\rho,V\in L^\infty(\Pi)$ and the $n\times n$ matrix $B\in L^\infty(\Pi)$
are real-valued and satisfy $\rho(x)\geq \underline{\rho}>0$ and $B(x)\geq \underline{B}\,I>0$ on $\Pi$, where $I$ is the unit matrix (whereas $V$ can have any sign in general).
Here $\Pi:=\mathbb{R}$ for $n=1$ or, for $n\geq 2$, $\Pi:=\mathbb{R}\times \Pi_{\widehat{1}}$ is the infinite parallelepiped, with $\Pi_{\widehat{1}}:=(0,X_2)\times\dots\times (0,X_n)$.
\vskip 1 pt
Also $D_t=\frac{\partial}{\partial t}$ and $D_i=\frac{\partial}{\partial x_i}$ are the partial derivatives, and the operators
$\dv$ and $\nabla$ are taken with respect to space variables.
\par We also assume that, for some (sufficiently large) $X_0>0$,
\begin{gather}
 \rho(x)=\rho_{\infty},\ \
 B(x)=\textrm{diag}\,(B_{1\infty},\ldots,B_{n\infty}),\ \
 V(x)=V_{\infty}\ \ \text{for}\ \ |x_1|\geq X_0,
 \label{eq:s21}
\end{gather}
where $\textrm{diag}\,(B_{1\infty},\ldots,B_{n\infty})$ is the diagonal matrix with the listed positive diagonal entries.
More generally, it could be easily assumed that $\rho$, $B$ and $V$ have different constant values for $x_1\leq -X_0$ and for $x_1\geq X_0$.
Let $X_1>X_0$, and $\Omega=\Omega_X=(-X_1,X_1)$ for $n=1$ or $\Omega=\Omega_X=(-X_1,X_1)\times \Pi_{\widehat{1}}$ for $n\geq 2$.
\vskip 1pt
\par We consider the weak solution $\psi\in C([0,T]; H_0^1(\Pi))$ having
$D_t\psi\in C([0,T];\linebreak L^2(\Pi))$ and satisfying the integral identity
\begin{equation}
 i\hbar(D_t\psi(\cdot,t),\varphi)_{L^{2,\rho}(\Pi)}
 =\mathcal{L}_{\Pi}(\psi(\cdot,t),\varphi)\ \ \text{for any}\ \ \varphi\in H_0^1(\Pi),\ \ \text{on}\ \ [0,T],
\label{weak1}
\end{equation}
and the initial condition $\psi|_{t=0}=\psi^0\in H_0^1(\Pi)$.
Hereafter we use the standard complex Lebesgue and Sobolev spaces (and subspaces),
the weighted complex Lebesgue space $L^{2,\rho}(G)$ endowed by the inner product
$(w,\varphi)_{L^{2,\rho}(G)}:=(\rho w,\varphi)_{L^2(G)}$
and the $\mathcal{H}$-related Hermitian-symmetric sesquilinear form
\[
 \mathcal{L}_G(w,\varphi)
 :=\textstyle{\frac{\hbar^{\,2}}{2}}(B\nabla w,\nabla\varphi)_{L^2(G)}+(Vw,\varphi)_{L^2(G)},\ \
 \text{with}\ \ G=\Pi,\Omega,\text{etc.}
\]
\par We define a non-uniform mesh in $x_1$ on $\mathbb{R}$ containing the points $\pm X_1$ and being uniform with a step $0<h_1<X_1$ outside $[-X_1+h_1,X_1-h_1]\supset [-X_0,X_0]$.
We also define non-uniform meshes in $x_2,\dots,x_n$ respectively on $[0,X_2],\ldots,[0,X_n]$ (containing the ends of the segments).
They induce the partition of $\bar{\Pi}$ into finite elements that are rectangular parallelepipeds without common internal points. Let $|h|$ be their maximal diagonal length.
Let $S_h(\bar{\Pi})$ be the (infinite-dimensional) subspace of functions in
$H_0^1(\Pi)$ that are polylinear over each element.
Clearly $S_h(\bar{\Pi})\subset C(\bar{\Pi})\cap L^2(\Pi)$.
Let $S_h$ be the restriction of $S_h(\bar{\Pi})$ to $\bar{\Omega}$.
\par Let $\overline{\omega}^{\,\tau}_M$ be the non-uniform mesh $0=t_0<\ldots<t_M=T$ with steps
$\tau_m:=t_m-t_{m-1}$.
We put $\tau_{\rm max}:=\max_{1\leq m\leq M}\tau_m$ and
$\hat{\tau}_m:=\frac{\tau_m+\tau_{m+1}}{2}$ for $1\leq m\leq M-1$ and $\hat{\tau}_0:=\frac{\tau_1}{2}$.
We define the time mesh operators
\[
 \overline{\partial}_tY^m:= \frac{Y^m-Y^{m-1}}{\tau_m},\ \
 \hat{\partial}_tY^m:= \frac{Y^{m+1}-Y^m}{\hat{\tau}_m},\ \
 \overline{s}_tY^m:= \frac{Y^{m-1}+Y^m}{2}.
\]
\par We introduce \textit{the Crank-Nicolson-polylinear FEM} approximate solution $\Psi$: $\overline{\omega}^{\,\tau}_M\to S_h(\bar{\Pi})$ satisfying the integral identity
\begin{equation}
 i\hbar(\overline{\partial}_t\Psi^m,\varphi)_{L^{2,\rho}(\Pi)}
 =\mathcal{L}_{\Pi}(\overline{s}_t\Psi^m,\varphi)
 \ \ \text{for any}\
 \varphi\in S_h(\bar{\Pi}) \ \text{and}\ 1\leq m\leq M,
\label{p43}
\end{equation}
compare with \eqref{weak1},
and the initial condition $\Psi|_{t=0}=\Psi^0\in S_h(\bar{\Pi})$,
where $\Psi^0$ is an approxi\-ma\-tion for $\psi^0$.
\par Let $\ell_\infty^m(\varphi)$ be a conjugate linear functional on $S_h(\bar{\Pi})$ that we add to the right-hand side of \eqref{p43} to study stability in more detail and to derive error estimates.
\begin{proposition}
Let $\ell_\infty^m(\varphi)=(F^m,\varphi)_{L^2(\Pi)}$ with $F^m\in L^2(\Pi)$ for $1\leq m\leq M$.
Then there exists a unique approximate solution $\Psi$ and the following first stability bound holds
\begin{equation}
 \max_{0\leq m\leq M} \|\Psi^m \|_{L^{2,\rho}(\Pi)}
 \leq \|\Psi^0\|_{L^{2,\rho}(\Pi)}
 +\frac{2}{\hbar}\sum_{m=1}^M \left\|F^m\right\|_{L^{2,1/\rho}(\Pi)}\tau_m.
 \label{p12}
\end{equation}
\label{prop:1Pi}
\end{proposition}
\par We introduce also the ``energy'' norm such that
\begin{gather}
\hspace{-6pt} \|w\|_{{\mathcal H}+\hat{v}\rho;\,\Pi}^2
 :={\mathcal L}_{\Pi}(w,w)+\hat{v}\|w\|^2_{L^{2,\rho}(\Pi)}
 \geq\hat{\delta}\|w\|^2_{L^{2,\rho}(\Pi)}
 \ \text{for any}\ w\in H_0^1(\Pi),
\label{p20inf}
\end{gather}
with some real numbers $\hat{v}$ and $\hat{\delta}>0$.
Inequality \eqref{p20inf} is knowingly valid for $\hat{v}$ so large that $\textstyle{\frac{\hbar^{\,2}}{2}}\underline{B}\lambda_0+V(x)+(\hat{v}-\hat{\delta})\rho(x)\geq 0$ on $\Omega$
with $\lambda_0:=\sum_{k=2}^n\bigl(\frac{\pi}{X_k}\bigr)^2$ (here $\lambda_0=0$ for $n=1$).
We define also the corresponding dual mesh depending norm
\[
 \|w\|_{H_h^{-1}(\Pi)}
 := \max_{\varphi\in S_h(\bar{\Pi}):\,\|\varphi\|_{{\mathcal H}+\hat{v}\rho;\,\Pi}=1}
 |\langle w,\varphi\rangle_{\Pi}|
 \leq c\|w\|_{H^{-1}(\Pi)},
\]
where $\langle w,\varphi\rangle_{\Pi}$ is the conjugate duality relation on
$H^{-1}(\Pi)\times H_0^1(\Pi)$ and $H^{-1}(\Pi)=[H_0^1(\Pi)]^*$.
Hereafter $c$ and $c_1$ are generic constants independent of the meshes, any functions and $T$ whereas $c_0$ denotes absolute constants (fixed numbers).
\begin{proposition}
Let $\ell_\infty^m(\varphi)=\langle F^m,\varphi\rangle_{\Pi}$
with $F^m\in  H^{-1}(\Pi)$ for $1\leq m\leq M$ and $F^0\in H^{-1}(\Pi)$ be arbitrary.
Then there exists a unique approximate solution $\Psi$ and the following second stability bound holds
\begin{gather}
\max_{0\leq m\leq M}\left\|\Psi^m\right\|_{\mathcal{H}+\hat{v}\rho;\,\Pi}
 \leq \left\| \Psi^0 \right\|_{{\mathcal H}+\hat{v}\rho;\,\Pi}
\nonumber\\
+4\sum_{m=1}^M \Bigl(\frac{|\hat{v}|}{\hbar}\,\|F^m\|_{H_h^{-1}(\Pi)}
 + \left\| \overline{\partial}_tF^m\right\|_{H_h^{-1}(\Pi)}\Bigr)\tau_m
 +4\left\| F^0\right\|^{(-1)}_h.
\label{p22}
\end{gather}
\label{prop:2Pi}
\end{proposition}
\par Method \eqref{p43} cannot be directly used in practice because of the infinite number of unknowns at each time level.
Nevertheless it is possible to restrict the method to $\bar{\Omega}$ provided that $\Psi^0\in S_{0h}:=\{\varphi\in S_h;\,\varphi(x)=0$ on $\Omega\setminus\Omega_0\}$, where $\Omega_0:=\Omega_{X_1-h_1,X_2,\ldots,X_n}$,
and $\overline{\omega}^{\,\tau}_M$ is uniform with the step $\tau=\frac{T}{M}$.
Let both assumptions be valid up to the end of the section.
\par By definition, \textit{the discrete TBCs} are conditions at the artificial boundaries $x_1=\pm X_1$ allowing to accomplish the restriction (they are non-local in $x_2,\ldots,x_n$ and $t$).
To write down them explicitly, for clarity, we confine ourselves by the case of the uniform mesh in $x_k$ with the step
$h_k=\frac{X_k}{J_k}$, for $2\leq k\leq n$, and define the related well-known direct and inverse discrete sine Fourier transforms
\begin{gather*}
 P^{(q)}=\left({\mathcal F}_kP\right)^{(q)}:= \frac{2}{J_k} \sum_{j=1}^{J_k-1} P_j \sin \frac{\pi q j}{J_k},\ \ 1\leq q\leq J_k-1,
\\
 P_j=\bigl({\mathcal F}_k^{-1}P^{(\cdot)}\bigr)_j:= \sum_{q=1}^{J_k-1} P^{(q)} \sin \frac{\pi q j}{J_k},\ \ 1\leq j\leq J_k-1.
\end{gather*}
The related eigenvalues of the 1D linear in $x_k$ FEM counterparts of the operators $-D_k^2$ and the unit one (for zero Dirichlet boundary values at $x_k=0,X_k$) are
\[
 \lambda_q^{(k)}=\Bigl( \frac{2}{h_k}\sin \frac{\pi q h_k}{2X_k}\Bigr)^2,\ \
 \sigma_q^{(k)}=1-\frac{2}{3}\sin^2 \frac{\pi q h_k}{2X_k}\in \Bigl(\frac13,1\Bigr).
\]
Denote by $\omega_{h\widehat{1}}$ the internal part of the introduced uniform mesh in $\bar{\Pi}_{\widehat{1}}$ and define the related mesh inner product
\begin{gather*}
 (U,W)_{\omega_{h\widehat{1}}}
 :=\sum_{j_2=1}^{J_2-1}\ldots\sum_{j_n=1}^{J_n-1}
 U_{j_2,\dots,j_n}W^*_{j_2,\ldots,j_n}h_2\dots h_n\ \ \text{for}\ \ n\geq 2
\end{gather*}
or set $(U,W)_{\omega_{h\widehat{1}}}:=UW^*$ for $n=1$,
where $W^*$ is the complex conjugate for $W$.
\vskip 1pt
\par Recall that the discrete convolution of mesh functions $R,Q$: $\overline{\omega}_M^{\,\tau}\to {\mathbb C}$ is given by
$(R*Q)^m:=\sum_{p=0}^m R^pQ^{m-p}$ for $0\leq m\leq M$.
\begin{proposition}
\label{Zlotnik_mini4:th:pk}
The restriction $\Psi|_{\bar{\Omega}}$ of the above approximate solution
obeys the integral identity on $\Omega$
\begin{gather}
 i\hbar(\overline{\partial}_t\Psi^m,\varphi)_{L^{2,\rho}(\Omega)}
 =\mathcal{L}_{\Omega}(\overline{s}_t\Psi^m,\varphi)
\nonumber\\[1mm]
 -\textstyle{\frac{\hbar^{\,2}}{2}}B_{1\infty}(\mathcal{S}_{\rm ref}^{m}\bPsi_{X_1}^m,\varphi|_{x_1=X_1})_{\omega_{h\widehat{1}}}
 +\textstyle{\frac{\hbar^{\,2}}{2}}B_{1\infty}(\mathcal{S}_{\rm ref}^{m}\bPsi_{-X_1}^m,\varphi|_{x_1=-X_1})_{\omega_{h\widehat{1}}}
\label{eq:ii2}
\end{gather}
for any $\varphi\in S_h$ and $1\leq m\leq M$,
and the initial condition
\begin{gather}
 \Psi|_{t=0}=\Psi^0|_{\bar{\Omega}}\in S_h.
\label{eq:icom}
\end{gather}
Here
${\mathbf \Psi}^m_{\pm X_1}=\bigl\{\bigl.{\Psi}^0\bigr|_{x_1=\pm X_1},\ldots,\bigl.{\Psi}^m\bigr|_{x_1=\pm X_1}\bigr\}$ is a vector-function.\\
\par The operator $\mathcal{S}_{\rm ref}$ in the discrete TBC has the form
\begin{equation}
 {\mathcal S}^m_{\rm ref} {\mathbf \Phi}^m:=
 {\mathcal F}_2^{-1}\dots {\mathcal F}_n^{-1}\left[
 \sigma_{q_2}^{(2)}\dots\sigma_{q_n}^{(n)}
 R_{\mathbf q}*\Phi^{\mathbf q}\right]^m\ \ \text{on}\ \ \overline{\omega}_M^{\,\tau}
\label{Zlotnik_mini4:eq:l7}
\end{equation}
for any $\Phi$: $\omega_{h\widehat{1}}\times \overline{\omega}_M^{\,\tau}\to {\mathbb C}$ such that $\Phi^0=0$, where
${\mathbf \Phi}^m:=\{\Phi^0,\dots,\Phi^m\}$,
$\Phi^{\mathbf q}:=({\mathcal F}_n\dots ({\mathcal F}_2\Phi)^{(q_2)}\dots)^{(q_n)}$ and ${\mathbf q}=(q_2,\dots,q_n)$.
Here the kernel $R_{\mathbf q}$ can be computed by the recurrent formulas
\begin{gather*}
 R_{\mathbf q}^0=c_{1\mathbf q},\,\
 R_{\mathbf q}^1=-c_{1\mathbf q}\kappa_{\mathbf q} \mu_{\mathbf q},\,\
 R_{\mathbf q}^m=\frac{2m-3}{m}\kappa_{\mathbf q} \mu_{\mathbf q} R_{\mathbf q}^{m-1}
 -\frac{m-3}{m}\kappa_{\mathbf q}^2 R_{\mathbf q}^{m-2},\ m\geq 2,
\end{gather*}
with the coefficients defined by
\begin{gather}
 c_{1\mathbf q}=-\frac{|\alpha_{\mathbf q}|^{1/2}}{2}e^{-i(\arg \alpha_{\mathbf q})/2},\ \
 \kappa_{\mathbf q}=-e^{i\arg \alpha_{\mathbf q}},\ \
 \mu_{\mathbf q}=\frac{\beta_{\mathbf q}}{|\alpha_{\mathbf q}|}\in(-1,1),
\label{eq: c1vkmu}\\
 \alpha_{\mathbf q}=2a_{\mathbf q}+\frac{1}{3}h_1^2a_{\mathbf q}^2\neq 0,\ \
 \arg \alpha_{\mathbf q}\in(0,2\pi),\ \
 \beta_{\mathbf q}=2\Rea a_{\mathbf q}+\frac{1}{3}h_1^2|a_{\mathbf q}|^2,
\nonumber\\
 a_{\mathbf q}=\frac{V_{\infty}}{B_{1\infty}\hbar^2}+\frac{1}{2B_{1\infty}}\Bigr(B_{2\infty}\frac{\lambda^{(2)}_{q_2}}{\sigma^{(2)}_{q_2}}+\dots +B_{n\infty}\frac{\lambda^{(n)}_{q_n}}{\sigma^{(n)}_{q_n}}\Bigr)
 +i\, \frac{2\rho_\infty}{\tau\hbar B_{1\infty}}.
\nonumber
\end{gather}
\end{proposition}
\par The next lemma is important to prove stability results for method \eqref{eq:ii2}, \eqref{eq:icom}.
\begin{lemma}
\label{lem1}
The operator ${\mathcal S}^m_{\rm ref}$ satisfies the inequalities \cite{DZ06,DZ07}
\begin{gather*}
 \Ima\sum_{l=1}^m \left({\mathcal S}^l_{\rm ref}{\mathbf \Phi}^l,\overline{s}_t \Phi^l\right)_{\omega_{h\widehat{1}}}\tau \geq 0,\ \
 \Ima\sum_{l=1}^m \left({\mathcal S}^l_{\rm ref}{\mathbf \Phi}^l,
 (i\hbar\overline{\partial}_t+\hat{v}\overline{s}_t)\Phi^l\right)_{\omega_{h\widehat{1}}}\tau \geq 0\ \
\end{gather*}
on $\overline{\omega}_M^{\,\tau}$,
for any $\Phi$: $\omega_{h\widehat{1}}\times\overline{\omega}_M^{\,\tau}\to {\mathbb C}$ such that
$\Phi^0=0$ and $\hat{v}\geq -\frac{V_\infty}{\rho_\infty}$ (see \eqref{eq:s21}).
\end{lemma}
\par Let $\ell^m(\varphi)$ be a conjugate linear functional on $S_h$ that we add to the right-hand side of \eqref{eq:ii2} to study stability in more detail.
\begin{proposition}
Let $\ell^m(\varphi)=(F^m,\varphi)_{L^2(\Omega)}$ with $F^m\in L^2(\Omega)$ for $1\leq m\leq M$.
Then the solution to \eqref{eq:ii2}, \eqref{eq:icom} is unique and satisfies the first stability bound
\begin{equation}
 \max_{0\leq m\leq M} \|\Psi^m\|_{L^{2,\rho}(\Omega)}
 \leq \|\Psi^0\|_{L^{2,\rho}(\Omega)}
 +\frac{2}{\hbar}\sum_{m=1}^M \left\|F^m\right\|_{L^{2,1/\rho}(\Omega)}\tau.
 \label{p12}
\end{equation}
\label{prop:1}
\end{proposition}
\par We introduce the ``energy'' norm on $\Omega$ such that
\begin{gather}
\|w\|_{{\mathcal H}+\hat{v}\rho;\,\Omega}^2
:={\mathcal L}_{\Omega}(w,w)+\hat{v}\left\|\sqrt{\rho}\,w\right\|^2_{L^2(\Omega)}>0,
\label{p20}
\end{gather}
for any $w\in \tilde{H}^1(\Omega):=\{H^1(\Omega);\, w|_{(-X_1,X_1)\times\partial\Pi_{\widehat{1}}}=0\}$ except for $w=0$, and some real number $\hat{v}\geq -\frac{V_\infty}{\rho_\infty}$.
In particular, for $\hat{v}$ so large that
$\textstyle{\frac{\hbar^{\,2}}{2}}\underline{B}\lambda_0+V(x)+\hat{v}\rho(x)>0$ on $\Omega$,
\eqref{p20} is valid.
Define also the respective dual mesh depending norm
\[
 \|w\|_{H_h^{-1}(\Omega)}
 := \max_{\varphi\in S_h:\,\|\varphi\|_{{\mathcal H}+\hat{v}\rho;\,\Omega}=1}
 |\langle w,\varphi\rangle_{\Omega}|
 \leq c\|w\|_{H^{-1}(\Omega)},\ \ H^{-1}(\Omega)=[\tilde{H}^1(\Omega)]^*,
\]
where $\langle w,\varphi\rangle_{\Omega}$ is the conjugate duality relation on
$H^{-1}(\Omega)\times \tilde{H}^1(\Omega)$.
\begin{proposition}
Let $\ell^m(\varphi)=\langle F^m,\varphi\rangle_{\Omega}$
with $F^m\in  H^{-1}(\Omega)$ for $1\leq m\leq M$ and $F^0\in H^{-1}(\Omega)$ be arbitrary.
Then the solution to \eqref{eq:ii2}, \eqref{eq:icom} is unique and satisfies the second stability bound
\begin{gather}
\max_{0\leq m\leq M}\left\|\Psi^m\right\|_{\mathcal{H}+\hat{v}\rho;\,\Omega}
 \leq \left\| \Psi^0 \right\|_{{\mathcal H}+\hat{v}\rho;\,\Omega}
\nonumber\\
+4\sum_{m=1}^M \Bigl(\frac{|\hat{v}|}{\hbar}\,\|F^m\|_{H_h^{-1}(\Omega)}
 + \left\| \overline{\partial}_tF^m\right\|_{H_h^{-1}(\Omega)}\Bigr)\tau
 +4\left\| F^0\right\|_{H_h^{-1}(\Omega)}.
\label{p22}
\end{gather}
\label{prop:2}
\end{proposition}
\par Propositions \ref{prop:1Pi}--\ref{prop:2} and Lemma \ref{lem1} in the quite similar cases of the semi-infinite $\Pi=(0,\infty)$ ($n=1$) and $\Pi=(0,\infty)\times (0,X_2)$ ($n=2$) were proved respectively in \cite{DZZ09} and \cite{ZZ11,IZ11} (see also \cite{IZ13}), where families of finite-difference schemes with space averages depending on a parameter $\theta$ were treated covering, in particular, the linear and bilinear FEMs (for $\theta=\frac16$).
For the presented improvement in formulas \eqref{eq: c1vkmu}, see also \cite{ZZ14a}.
For $n=1$, the results are as well particular cases of those from \cite{ZZ12} (given specifically for general FEM).
The case $n\geq 3$ can be treated in the same manner as $n=2$ (for such an example, see \cite{ZDZRPROC14}).
\par The numerical results for the method can be found in \cite{DZZ09,IZ11,IZ13}.

\section{Error estimates}
\label{errest}
\par Let condition \eqref{p20inf} be valid and $\sigma w$ be \textit{the elliptic projection} of $w\in H_0^1(\Pi)$ onto $S_h(\bar{\Pi})$ such that
\begin{gather}
  \mathcal{L}_{\Pi}(\sigma w,\varphi)+\hat{v}(\sigma w,\varphi)_{L^{2,\rho}(\Pi)}
 =\mathcal{L}_{\Pi}(w,\varphi)+\hat{v}(w,\varphi)_{L^{2,\rho}(\Pi)}
\label{proid}
\end{gather}
for any $\varphi\in S_h(\bar{\Pi})$. Note that $\sigma w$ exists and is unique. We also assume below that $B\in W^{1,\infty}(\Pi)$ and then the following error estimate holds
\begin{gather}
\hspace{-6pt} \|w-\sigma w\|_{L^{2,\rho}(\Pi)}\leq c|h|^2\|(\mathcal{H}_\rho+\hat{v})w\|_{L^{2,\rho}(\Pi)}\
 \text{for any}\ w\in H^2(\Pi)\cap H_0^1(\Pi).
\label{errappr}
\end{gather}
\par We consider $\mathcal{H}_\rho:=\frac{1}{\rho}\mathcal{H}$ as an unbounded operator in $L^{2,\rho}(\Pi)$ with $\mathcal{D}(\mathcal{H}_\rho)=H^2(\Pi)\cap H_0^1(\Pi)$.
Assume below that $\psi^0\in \mathcal{D}(\mathcal{H}_\rho^3)$.
\begin{proposition}
The following first error estimate holds
\begin{gather}
 \max_{0\leq m\leq M}\|(\psi-\Psi)^m\|_{L^{2,\rho}(\Pi)}
 \leq\|\Psi^0-\sigma\psi^0\|_{L^{2,\rho}(\Pi)}
\nonumber\\[1mm]
 +c(1+T)\left\{\tau_{\rm max}^2\|\mathcal{H}_\rho^3\psi^0\|_{L^{2,\rho}(\Pi)}
 +|h|^2\bigl(\|\mathcal{H}_\rho^2\psi^0\|_{L^{2,\rho}(\Pi)}
       +\|\psi^0\|_{L^{2,\rho}(\Pi)}\bigr)
 \right\}.
\label{err12}
\end{gather}
Here $\sigma\psi^0$ can be replaced by $\psi^0$.
\label{prop:1errPi}
\end{proposition}
\begin{proof}
1. For $y\in L^1(0,T)$, define the average (the projection on the time mesh)
$ [y]^m:=\frac{1}{\tau_m}\int_{t_{m-1}}^{t_m}y(t)dt,\ \ 1\leq m\leq M,$
and notice that
\begin{gather}
 \sum_{m=1}^M\|[u]^m\|_\mathbf{B}\tau_m\leq\int_0^{T}\|u(\cdot,t)\|_\mathbf{B}\,dt,\ \
\label{ineq2a}
\end{gather}
where $\|\cdot\|_\mathbf{B}=\|\cdot\|_{L^{2,\rho}(\Pi)}$ or $\|\cdot\|_{\mathcal{H}+\hat{v}\rho;\,\Pi}$ and $u\in L^1(0,T;\mathbf{B})$,
and
\begin{gather}
 |[y]^m-\overline{s}_ty^m|\leq c_0\tau_m^2[|D_t^2y|]^m,\ \ 1\leq m\leq M,\ \ \text{for}\ \ y\in W^{2,1}(0,T).
\label{ineq2b}
\end{gather}
\par 2. Applying $[\cdot]$ to identity \eqref{weak1}, we get
\begin{gather}
\hspace{-7pt} i\hbar([D_t\psi]^m,\varphi)_{L^{2,\rho}(\Pi)}
 =\mathcal{L}_{\Pi}([\psi]^m,\varphi)\ \ \text{for any}\
 \varphi\in H_0^1(\Pi)\ \text{and}\ 1\leq m\leq M.
\label{weak1avg}
\end{gather}
Then for any $\eta$: $\overline{\omega}^{\,\tau}_M\to S_h(\bar{\Pi})$ from identities \eqref{p43} and \eqref{weak1avg} it follows that
\begin{gather*}
 i\hbar(\overline{\partial}_t(\Psi-\eta)^m,\varphi)_{L^{2,\rho}(\Pi)}-\mathcal{L}_{\Pi}(\overline{s}_t(\Psi-\eta)^m,\varphi)
\nonumber\\[1mm]
 =i\hbar(([D_t\psi]-\overline{\partial}_t\eta)^m,\varphi)_{L^{2,\rho}(\Pi)}
 -\mathcal{L}_{\Pi}(([\psi]-\overline{s}_t\eta)^m,\varphi)
 \ \ \text{for any}\
 \varphi\in S_h(\bar{\Pi}).
\end{gather*}
\par Let $\eta^m:=\sigma\psi^m$, $0\leq m\leq M$. By identity \eqref{proid} and $[D_ty]=\overline{\partial}_ty$ we get
\begin{gather}
 i\hbar(\rho\overline{\partial}_t(\Psi-\sigma\psi)^m,\varphi)_{L^2(\Pi)}-\mathcal{L}_{\Pi}(\overline{s}_t(\Psi-\sigma\psi)^m,\varphi)
\nonumber\\[1mm]
 =i\hbar(\rho([D_t\psi]^m-[D_t\sigma\psi]^m),\varphi)_{L^2(\Pi)}
 -\mathcal{L}_{\Pi}(([\psi]-\overline{s}_t\psi)^m,\varphi)
\nonumber\\[1mm]
 +\hat{v}(\rho\overline{s}_t(\psi-\sigma\psi)^m,\varphi)_{L^2(\Pi)}
 =:(F,\varphi)_{L^2(\Pi)}\ \ \text{for any}\ 1\leq m\leq M,
\label{errid2}
\end{gather}
where (after rearranging the summands)
\begin{gather}
 F=-([\mathcal{H}\psi]-\overline{s}_t\mathcal{H}\psi)
 +i\hbar\rho[D_t(\psi-\sigma\psi)]+\hat{v}\rho\overline{s}_t(\psi-\sigma\psi).
\label{errF}
\end{gather}
\par Let now on $\hbar=1$. Proposition \ref{prop:1Pi} together with \eqref{ineq2a} lead to the bound
\begin{gather*}
 \max_{0\leq m\leq M}\|(\Psi-\sigma\psi)^m\|_{L^{2,\rho}(\Pi)}
 \leq\|\Psi^0-\sigma\psi^0\|_{L^{2,\rho}(\Pi)}
 +2\sum_{m=1}^M \left\|F^m\right\|_{L^{2,1/\rho}(\Pi)}\tau_m
\\
 \leq\|\Psi^0-\sigma\psi^0\|_{L^{2,\rho}(\Pi)}
+2\sum_{m=1}^M\|\bigl([\mathcal{H}_\rho\psi]-\overline{s}_t\mathcal{H}_\rho\psi\bigr)^m\|_{L^{2,\rho}(\Pi)}\tau_m
\\
 +2\int_0^T\|D_t(\psi-\sigma\psi)\|_{L^{2,\rho}(\Pi)}\,dt
 +2|\hat{v}|T\max_{0\leq m\leq M}\|(\psi-\sigma\psi)^m\|_{L^{2,\rho}(\Pi)}.
\end{gather*}
The formula $\psi-\Psi=\psi-\sigma\psi-(\Psi-\sigma\psi)$ and estimates \eqref{errappr} and \eqref{ineq2b} imply
\begin{gather}
 \max_{0\leq m\leq M}\|(\psi-\Psi)^m\|_{L^{2,\rho}(\Pi)}
 \leq\|\Psi^0-\sigma\psi^0\|_{L^{2,\rho}(\Pi)}
\nonumber\\
 +c\Bigl\{\tau_{\rm max}^2\int_0^T\|D_t^2\mathcal{H}_\rho\psi\|_{L^{2,\rho}(\Pi)}\,dt
 +|h|^2\int_0^T\|D_t(\mathcal{H}_\rho+\hat{v})\psi\|_{L^{2,\rho}(\Pi)}\,dt
\nonumber\\
 +(1+|\hat{v}|T)|h|^2\max_{0\leq m\leq M}\|(\mathcal{H}_\rho+\hat{v})\psi^m\|_{L^{2,\rho}(\Pi)}
\Bigr\}.
\label{errest2}
\end{gather}
\par Under the above assumptions, the solution to problem \eqref{eq:se}, \eqref{eq:ic} satisfies the bound
\begin{gather}
 \max_{0\leq t\leq T}\|D_t^k\psi\|_{L^{2,\rho}(\Pi)}\leq\|(D_t^k\psi)|_{t=0}\|_{L^{2,\rho}(\Pi)}
 =\|\mathcal{H}_\rho^k\psi^0\|_{L^{2,\rho}(\Pi)},\ \ 0\leq k\leq 3,
\label{solb1}
\end{gather}
and the property $D_t^k\psi=D_t^{k-l}(-i\mathcal{H}_\rho)^l\psi$ for $1\leq l\leq k$. Therefore
\begin{gather*}
 \max_{0\leq m\leq M}\|(\psi-\Psi)^m\|_{L^{2,\rho}(\Pi)}
\leq\|\Psi^0-\sigma\psi^0\|_{L^{2,\rho}(\Pi)}
 +c\left\{T\tau_{\rm max}^2\|\mathcal{H}_\rho^3\psi^0\|_{L^{2,\rho}(\Pi)}
\right.
\nonumber\\[1mm]
\left.
 +T|h|^2\|\mathcal{H}_\rho^2\psi^0\|_{L^{2,\rho}(\Pi)}
 +(1+|\hat{v}|T)|h|^2\bigl(\|\mathcal{H}_\rho\psi^0\|_{L^{2,\rho}(\Pi)}
 +|\hat{v}|\|\psi^0\|_{L^{2,\rho}(\Pi)}\bigr)
  \right\}.
\end{gather*}
Note that for $\hat{v}=0$ the estimate is simplified.
\par The following multiplicative inequality holds
\begin{gather}
 \|\mathcal{H}_\rho^l\psi^0\|_{L^{2,\rho}(\Pi)}^2
\leq\|\mathcal{H}_\rho^{l+1}\psi^0\|_{L^{2,\rho}(\Pi)}\|\mathcal{H}_\rho^{l-1}\psi^0\|_{L^{2,\rho}(\Pi)}\ \ \text{for}\ \ l=1,2.
\label{ineq5}
\end{gather}
Using \eqref{errappr} for $w=\psi^0$ together with \eqref{ineq5} for $l=1$, we complete the proof.
\qed
\end{proof}
\begin{corollary}
Let $\psi^0(x)=0$ for $|x_1|\geq X_0$,
$\Psi^0\in S_{0h}$ and $\overline{\omega}^{\,\tau}_M$ be uniform.
Then for the solution to \eqref{eq:ii2}, \eqref{eq:icom} the following first error estimate holds
\begin{gather*}
 \max_{0\leq m\leq M}\|(\psi-\Psi)^m\|_{L^{2,\rho}(\Omega)}
 \leq\|\Psi^0-\psi^0\|_{L^{2,\rho}(\Omega)}
\nonumber\\
 +c(1+T)\left\{\tau^2\|\mathcal{H}_\rho^3\psi^0\|_{L^{2,\rho}(\Omega)}
 +|h|^2\bigl(\|\mathcal{H}_\rho^2\psi^0\|_{L^{2,\rho}(\Omega)}
       +\|\psi^0\|_{L^{2,\rho}(\Omega)}\bigr)
 \right\}.
\label{perr12}
\end{gather*}
\end{corollary}
\par The result immediately follows from Proposition \ref{prop:1errPi}.
Notice also that, for $1\leq n\leq 3$, for the interpolant $s\psi^0$ in $S_{0h}$ for $\psi^0$, the following error estimate holds
\[
 \|s\psi^0-\psi^0\|_{L^{2,\rho}(\Omega)}
\leq c|h|^2\bigl(\|\mathcal{H}_\rho\psi^0\|_{L^{2,\rho}(\Omega)}+\|\psi^0\|_{L^{2,\rho}(\Omega)}\bigr),
\]
thus one can set simply $\Psi^0:=s\psi^0$. Other possible choices of $\Psi^0$ with the same error estimate, for any $n\geq 1$, are the $L^2(\Omega_0)$ (possibly with a weight like $\rho$) projection of $\psi^0$ onto $S_{0h}$ or its elliptic projection onto $S_{0h}$ like \eqref{proid} (with $\Pi$ replaced by $\Omega_0$ and any $\varphi\in S_{0h}$).
\begin{remark}
Importantly, all the above results can be essentially generalized rather easily. For example, in the 2D case, the problem in an unbounded domain $\Pi$ of general shape with smooth boundary and several half-strip-like outlets to infinity can be treated by using combined linear triangle elements inside $\Pi$ except outlets and bilinear rectangular elements inside outlets.
\end{remark}
\vskip 1pt
\par Let $B\in W^{2,\infty}(\Pi)$ and $\rho,V\in W^{1,\infty}(\Pi)$.
\begin{proposition}
Let $\mathcal{H}_\rho^3\psi^0\in H_0^1(\Pi)$, $1\leq n\leq 3$ and $\Psi^0:=s\psi^0$.
Then the following second error estimate holds
\begin{gather}
\max_{0\leq m\leq M}\|s\psi^m-\Psi^m\|_{\mathcal{H}+\hat{v}\rho;\,\Pi}
 \leq c(1+T)\Bigl\{\tau_{\rm max}^2\bigl(\|\mathcal{H}_\rho^3\psi^0\|_{\mathcal{H}+\hat{v}\rho;\,\Pi}
 +\|\mathcal{H}_\rho^2\psi^0\|_{\mathcal{H}+\hat{v}\rho;\,\Pi}\bigr)
\nonumber\\
 +|h|^2\bigl(\|\mathcal{H}_\rho^3\psi^0\|_{L^{2,\rho}(\Pi)}+\|\psi^0\|_{L^{2,\rho}(\Pi)}\bigr)
 \Bigr\}.
\label{p22}
\end{gather}
Here $s\psi^m$ is the interpolant in $S_h(\bar{\Pi})$ for $\psi^m$, $0\leq m\leq M$.
\label{prop:2errPi}
\end{proposition}
\begin{proof}
1. Let first $n\geq 1$.
Inequality \eqref{p20inf} implies
\begin{gather}
 \|\rho w\|_{H_h^{-1}(\Pi)}\leq\hat{\delta}^{-1/2}\|w\|_{L^{2,\rho}(\Pi)},\ \
 \|w\|_{L^{2,\rho}(\Pi)}\leq\hat{\delta}^{-1/2}\|w\|_{\mathcal{H}+\hat{v}\rho;\,\Pi}.
\label{p24err}
\end{gather}
Then setting $\hat{c}:=1+\frac{|\hat{v}|}{\hat{\delta}}$, we also get
\begin{gather}
 \|\mathcal{H}w\|_{H_h^{-1}(\Pi)}
 \leq\|(\mathcal{H}+\hat{v}\rho)w\|_{H_h^{-1}(\Pi)}+|\hat{v}|\|\rho w\|_{H_h^{-1}(\Pi)}
 \leq \hat{c}\|w\|_{\mathcal{H}+\hat{v}\rho;\,\Pi}.
\label{p26err}
\end{gather}
\par For $y\in L^1(0,T)$, define two more averages (projections on the time mesh)
\begin{gather*}
 [y]_2^m:=\frac{1}{2\hat{\tau}_m}\int_{t_{m-1}}^{t_{m+1}}y(t)dt,\ \
 \langle y\rangle^m:=\frac{1}{\hat{\tau}_m}\int_{t_{m-1}}^{t_{m+1}}y(t)e_m(t)dt,\ \ 1\leq m\leq M-1,
\\
 [y]_2^0:=[y]^1,\ \
 \langle y\rangle^0:=\frac{2}{\tau_1}\int_0^{t_1}y(t)e_0(t)dt,
\end{gather*}
where $e_m(t)$ is the ``hat'' function linear on all segments $[t_{l-1},t_l]$ and such that
$e_m(t_m)=1$ and $e_m(t_l)=0$ for all $l\neq m$.
\vskip 1pt
\par Notice that the following relations hold
\begin{gather}
 \hat{\partial}_t[y]^m=\langle D_ty\rangle^m,\ \ (\hat{\partial}_t\overline{s}_ty)^m=[D_ty]_2^m,\ \ 1\leq m\leq M-1,
\label{p28err}\\
 [y]^1-y(0)=\frac{\tau_1}{2}\langle D_ty\rangle^0,\ \ (\overline{s}_ty)^1-y(0)=\frac{\tau_1}{2}[D_ty]^1,\ \ \text{for}\ \
 y\in W^{1,1}(0,T),
\label{p30err}\\
\sum_{m=0}^{M-1}\|\langle u\rangle^m\|\hat{\tau}_m\leq\int_0^T\|u(\cdot,t)\|\,dt,\ \
 \sum_{m=0}^{M-1}\|[u]_2^m\|\hat{\tau}_m\leq\int_0^T\|u(\cdot,t)\|\,dt,
\label{p32err}
\end{gather}
where $\|\cdot\|=\|\cdot\|_{L^{2,\rho}(\Pi)}$ and $u\in L^{2,1}(\Pi\times (0,T))$, and
\begin{gather}
 |\langle y\rangle^m-[y]_2^m|\leq c_0\tau_m^2[|D_t^2y|]_2^m,\ \ 1\leq m\leq M-1,\ \ \text{for}\ \ y\in W^{2,1}(0,T).
\label{p34err}
\end{gather}
\smallskip\par 2. Let first $\Psi^0\in S_h(\bar{\Pi})$ be arbitrary. We go back to the error identity \eqref{errid2}.
Applying Proposition \ref{prop:2Pi}, now we get
\begin{gather}
\max_{0\leq m\leq M}\|(\Psi-\sigma\psi)^m\|_{\mathcal{H}+\hat{v}\rho;\,\Pi}
 \leq \|\Psi^0-\sigma\psi^0\|_{{\mathcal H}+\hat{v}\rho;\,\Pi}
\nonumber\\
 +4\sum_{m=0}^{M-1}\| \hat{\partial}_tF^m\|_{H_h^{-1}(\Pi)}\hat{\tau}_m
 +4|\hat{v}|\sum_{m=1}^M\|F^m\|_{H_h^{-1}(\Pi)}\tau_m
 +4\|F^0\|_{H_h^{-1}(\Pi)},
\label{p22err}
\end{gather}
where the right-hand side is slightly transformed and $F$ is given by \eqref{errF}. We introduce the decomposition
\[
 F=F_\tau+\rho F_h,\ \
 F_\tau:=-([\mathcal{H}\psi]-\overline{s}_t\mathcal{H}\psi),\ \ F_h:=i[D_t(\psi-\sigma\psi)]+\hat{v}\overline{s}_t(\psi-\sigma\psi)
\]
as well as set $F_\tau^0:=0$ and $F_h^0=i D_t(\psi-\sigma\psi)|_{t=0}+\hat{v}(\psi^0-\sigma\psi^0)$.
\vskip 2pt
\par Applying sequentially relations \eqref{p26err}, \eqref{p28err}, \eqref{ineq2b}, \eqref{p34err}, \eqref{ineq2a} and \eqref{p32err}, we obtain
\begin{gather*}
 S_\tau:=|\hat{v}|\sum_{m=1}^M\|F_\tau^m\|_{H_h^{-1}(\Pi)}\tau_m
 +\sum_{m=0}^{M-1}\|\hat{\partial}_t F_\tau^m\|_{H_h^{-1}(\Pi)}\hat{\tau}_m
\nonumber\\
 \leq \hat{c}\Bigl(|\hat{v}|\sum_{m=1}^M\|([\psi]-\overline{s}_t\psi)^m\|_{\mathcal{H}+\hat{v}\rho;\,\Pi}\tau_m
 +\sum_{m=1}^{M-1}\|(\langle D_t\psi\rangle-[D_t\psi]_2)^m\|_{\mathcal{H}+\hat{v}\rho;\,\Pi}\hat{\tau}_m
\nonumber\\
 +\|([\psi]-\overline{s}_t\psi)^1\|_{\mathcal{H}+\hat{v}\rho;\,\Pi}\Bigr)
\nonumber\\
 \leq \hat{c}c_0\tau_{\rm max}^2\Bigl(|\hat{v}|\sum_{m=1}^M\|[|D_t^2\psi|]^m\|_{\mathcal{H}+\hat{v}\rho;\,\Pi}\tau_m
 +\sum_{m=1}^{M-1}\|[|D_t^3\psi|]_2^m\|_{\mathcal{H}+\hat{v}\rho;\,\Pi}\hat{\tau}_m
\nonumber\\
 +\|[|D_t^2\psi|]^1\|_{\mathcal{H}+\hat{v}\rho;\,\Pi}\Bigr)
\end{gather*}
\begin{gather}
 \leq \hat{c}c_0\tau_{\rm max}^2\Bigl\{\|(D_t^2\psi)|_{t=0}\|_{\mathcal{H}+\hat{v}\rho;\,\Pi}
 +\int_0^T\bigl(|\hat{v}|\|D_t^2\psi\|_{\mathcal{H}+\hat{v}\rho;\,\Pi}
 +2\|D_t^3\psi\|_{\mathcal{H}+\hat{v}\rho;\,\Pi}\bigr)\,dt\Bigr\}.
\label{p36err}
\end{gather}
\par The left inequality \eqref{p24err} implies
\begin{gather}
 S_h:=
 |\hat{v}|\sum_{m=1}^M\|\rho F_h^m\|_{H_h^{-1}(\Pi)}\tau_m
 +\|\rho F_h^0\|_{H_h^{-1}(\Pi)}
 +\sum_{m=0}^{M-1}\|\rho \hat{\partial}_tF_h^m\|_{H_h^{-1}(\Pi)}\hat{\tau}_m
\nonumber\\
 \leq \hat{\delta}^{-1/2}\Bigl(|\hat{v}|\sum_{m=1}^M\|F_h^m\|_{L^{2,\rho}(\Pi)}\tau_m
 +\|F_h^0\|_{L^{2,\rho}(\Pi)}
 +\sum_{m=0}^{M-1}\|\hat{\partial}_tF_h^m\|_{L^{2,\rho}(\Pi)}\hat{\tau}_m\Bigr),
\label{p38err}
\end{gather}
and further the error estimate \eqref{errappr} leads to
\begin{gather}
 \|F_h^0\|_{L^{2,\rho}(\Pi)}
 \leq \|D_t(\psi-\sigma\psi)|_{t=0}\|_{L^{2,\rho}(\Pi)}+|\hat{v}|\|\psi^0-\sigma\psi^0\|_{L^{2,\rho}(\Pi)}
\nonumber\\
 \leq c|h|^2\bigl(\|D_t(\mathcal{H}_\rho+\hat{v})\psi|_{t=0}\|_{L^{2,\rho}(\Pi)}
 +|\hat{v}|\|(\mathcal{H}_\rho+\hat{v})\psi^0\|_{L^{2,\rho}(\Pi)}\bigr).
\label{p40err}
\end{gather}
Applying sequentially relations \eqref{p28err}, \eqref{p30err}, \eqref{errappr} and \eqref{p32err}, we also obtain
\begin{gather}
 \sum_{m=0}^{M-1}\|\hat{\partial}_tF_h^m\|_{L^{2,\rho}(\Pi)}\hat{\tau}_m
\nonumber\\
 \leq\sum_{m=0}^{M-1}\bigl(\|\langle D_t^2(\psi-\sigma\psi)\rangle^m\|_{L^{2,\rho}(\Pi)}
 +|\hat{v}|\|[D_t(\psi-\sigma\psi)]_2^m\|_{L^{2,\rho}(\Pi)}\bigr)\hat{\tau}_m
\nonumber\\
 \leq c|h|^2\sum_{m=0}^{M-1}\bigl(\|\langle D_t^2(\mathcal{H}_\rho+\hat{v})\psi\rangle^m\|_{L^{2,\rho}(\Pi)}
 +|\hat{v}|\|[D_t(\mathcal{H}_\rho+\hat{v})\psi]_2^m\|_{L^{2,\rho}(\Pi)}\bigr)\hat{\tau}_m
\nonumber\\
 \leq c|h|^2\int_0^T\bigl(\|D_t^2(\mathcal{H}_\rho+\hat{v})\psi\|_{L^{2,\rho}(\Pi)}
 +|\hat{v}|\|D_t(\mathcal{H}_\rho+\hat{v})\psi\|_{L^{2,\rho}(\Pi)}\bigr)\,dt.
\label{p42err}
\end{gather}
Inserting into \eqref{p22err} all the estimates \eqref{p36err}-\eqref{p42err} together with the estimate for $\sum_{m=1}^M\|F_h^m\|_{L^{2,\rho}(\Pi)}\tau_m$ used in the preceding proof in \eqref{errest2}, we derive
\begin{gather*}
 \max_{0\leq m\leq M}\|(\Psi-\sigma\psi)^m\|_{\mathcal{H}+\hat{v}\rho;\,\Pi}
 \leq \|\Psi^0-\sigma\psi^0\|_{{\mathcal H}+\hat{v}\rho;\,\Pi}+ 4S_\tau+4S_h
\\
 \leq c\tau_{\rm max}^2\Bigl\{\|(D_t^2\psi)|_{t=0}\|_{\mathcal{H}+\hat{v}\rho;\,\Pi}
 +\int_0^T\bigl(|\hat{v}|\|D_t^2\psi\|_{\mathcal{H}+\hat{v}\rho;\,\Pi}
 +\|D_t^3\psi\|_{\mathcal{H}+\hat{v}\rho;\,\Pi}\bigr)\,dt\Bigr\}
\\
 +c|h|^2\Bigl\{|\hat{v}|\int_0^T\|D_t(\mathcal{H}_\rho+\hat{v})\psi\|_{L^{2,\rho}(\Pi)}\,dt
 +\hat{v}^2T\max_{0\leq t\leq T}\|(\mathcal{H}_\rho+\hat{v})\psi^m\|_{L^{2,\rho}(\Pi)}
\\
 +\|D_t(\mathcal{H}_\rho+\hat{v})\psi|_{t=0}\|_{L^{2,\rho}(\Pi)}
 +|\hat{v}|\|(\mathcal{H}_\rho+\hat{v})\psi^0\|_{L^{2,\rho}(\Pi)}
\\
 +\int_0^T\bigl(\|D_t^2(\mathcal{H}_\rho+\hat{v})\psi\|_{L^{2,\rho}(\Pi)}
 +|\hat{v}|\|D_t(\mathcal{H}_\rho+\hat{v})\psi\|_{L^{2,\rho}(\Pi)}\bigr)\,dt\Bigr\}.
\end{gather*}
Once again for $\hat{v}=0$ the estimate is essentially simplified.
\par Under all the above assumptions, the solution to problem \eqref{eq:se}, \eqref{eq:ic} satisfies the bound
\begin{gather*}
 \max_{0\leq t\leq T}\|D_t^k\psi\|_{\mathcal{H}+\hat{v}\rho;\,\Pi}\leq\|(D_t^k\psi)|_{t=0}\|_{\mathcal{H}+\hat{v}\rho;\,\Pi}
 =\|\mathcal{H}_\rho^k\psi^0\|_{\mathcal{H}+\hat{v}\rho;\,\Pi},\ \ 0\leq k\leq 3.
\end{gather*}
This bound and \eqref{solb1} and the property $D_t^k\psi=D_t^{k-l}(-i\mathcal{H}_\rho)^l\psi$, $1\leq l\leq k$,
imply
\begin{gather}
 \max_{0\leq m\leq M}\|(\Psi-\sigma\psi)^m\|_{\mathcal{H}+\hat{v}\rho;\,\Pi}
 \leq \|\Psi^0-\sigma\psi^0\|_{{\mathcal H}+\hat{v}\rho;\,\Pi}
\nonumber\\
 \leq c(1+T)\Bigl\{\tau_{\rm max}^2\bigl(\|\mathcal{H}_\rho^2\psi^0\|_{\mathcal{H}+\hat{v}\rho;\,\Pi}
 +\|\mathcal{H}_\rho^3\psi^0\|_{\mathcal{H}+\hat{v}\rho;\,\Pi}\bigr)
 +|h|^2\sum_{k=0}^3\|\mathcal{H}_\rho^k\psi^0\|_{L^{2,\rho}(\Pi)}
 \Bigr\}.
\label{p44err}
\end{gather}
\par 3. Let now $1\leq n\leq 3$ and $\Psi^0=s\psi^0$. Similarly to \cite{ZNMA94}, Lemma 5.1 for $n=1$ and
\cite{ZCMAM02}, Theorem 2.1 for $n=2$ and 3 (see also \cite{ZSMD83}), the following elliptic FEM error estimate holds
\begin{gather*}
 \|sw-\sigma w\|_{\mathcal{H}+\hat{v}\rho;\,\Pi}
 \leq c|h|^2\Bigl(\sum_{p\neq q}\|D_p^2D_qw\|_{L^2(\Pi)}+\sum_{p=1}^n\|D_p^2w\|_{L^2(\Pi)}\Bigr)
\\
 \leq c_1|h|^2\bigl(\|\mathcal{H}_\rho w\|_{\mathcal{H}+\hat{v}\rho;\,\Pi}+\|w\|_{\mathcal{H}+\hat{v}\rho;\,\Pi}\bigr)\ \
\end{gather*}
for any $w\in\mathcal{D}(\mathcal{H}_\rho)$ such that $\mathcal{H}_\rho w\in H_0^1(\Pi)$,
taking into account the above regularity assumptions on $B,\rho$ and $V$, where the first sum is taken over all $p$ and $q$ from 1 to $n$ excluding $p=q$ and disappears for $n=1$.
This estimate allows to pass from \eqref{p44err} to the final estimate \eqref{p22} by the triangle inequality together with inequalities \eqref{ineq5} and
\[
 \|\mathcal{H}_\rho^l\psi^0\|_{\mathcal{H}+\hat{v}\rho;\,\Pi}^2
 \leq \|(\mathcal{H}_\rho+\hat{v})\mathcal{H}_\rho^l\psi^0\|_{L^{2,\rho}(\Pi)}\|\mathcal{H}_\rho^l\psi^0\|_{L^{2,\rho}(\Pi)},\ \ l=0,1.
\ \ \ \ \ \ \ \ \qed
\]
\end{proof}
\begin{corollary}
Let $\psi^0(x)=0$ for $|x_1|\geq X_0$, $1\leq n\leq 3$ and $\Psi^0=s\psi^0$
on $\bar{\Omega}$
and $\overline{\omega}^{\,\tau}_M$ be uniform.
Then for the solution to \eqref{eq:ii2}, \eqref{eq:icom} the following second error estimate holds
\begin{gather*}
\max_{0\leq m\leq M}\|s\psi^m-\Psi^m\|_{\mathcal{H}+\hat{v}\rho;\,\Omega}
 \leq c(1+T)\Bigl\{\tau^2\bigl(\|\mathcal{H}_\rho^3\psi^0\|_{\mathcal{H}+\hat{v}\rho;\,\Omega}
 +\|\mathcal{H}_\rho^2\psi^0\|_{\mathcal{H}+\hat{v}\rho;\,\Omega}\bigr)
\nonumber\\
 +|h|^2\bigl(\|\mathcal{H}_\rho^3\psi^0\|_{L^{2,\rho}(\Omega)}+\|\psi^0\|_{L^{2,\rho}(\Omega)}\bigr)
 \Bigr\}.
\end{gather*}
\end{corollary}
\par The result immediately follows from Proposition \ref{prop:2errPi}.
\par Note that the norm $\|\cdot\|_{\mathcal{H}+\hat{v}\rho;\,\Omega}$ is equivalent to $\|\cdot\|_{H^1(\Omega)}$ and $\|s\cdot\|_{H^1(\Omega)}$ is actually the mesh counterpart of the latter norm.
\medskip\par\noindent\textbf{Acknowledgments}.
The study is supported by The National Research University -- Higher School of Economics' Academic Fund Program in 2014-2015,
research grant No. 14-01-0014.

\end{document}